\documentclass[12pt]{article}

\usepackage{amsmath,amsfonts,amsthm,amssymb}
\usepackage{graphicx}
\usepackage{url}

\newtheorem{theorem}{Theorem}

\theoremstyle{definition}

\newtheorem*{example}{Example}

\title{Counting Divisions of a $2\times n$ Rectangular Grid}
\author{Jacob Brown}

\date{May 2021}

\begin{document}

\maketitle

%\begin{biog} %comment out in initial submission to allow for double blind reviewing
%\item[\biogpic{\includegraphics[width=84pt]{me.jpg}}Jacob Brown] (jbrown18@ithaca.edu) studies math and physics at Ithaca College. He worked on this project as part of his ``Research Experience in Math" course during the spring of his junior year. If all goes well, he plans on going to graduate school for applied math, and hopes to one day become a professor. Outside of STEM, Jake has been a member of Ithaca College's mock trial team ever since his first year, and is currently a training director for the team.
%\end{biog}

Consider a $2\times n$ rectangle composed of $1\times 1$ squares. Suppose it is printed on a piece of paper that can be cut out with scissors. Cutting only along the edges between squares, how many ways are there to divide this rectangular board into two connected pieces? Three pieces? Four or five? What about $k$ pieces?

Durham and Richmond \cite{durham} showed that the number of ways to cut a $2\times n$ rectangle into two pieces is $2n^2-n$ and into three pieces is $\frac{2}{3}n^4-\frac{4}{3}n^3+\frac{11}{6}n^2-\frac{13}{6}n+1$. Their proofs centered around breaking the board up into $2\times 1$ rectangles and considering how each rectangle could be split into each piece. This enabled them to derive the formulas above, but became too complicated for greater numbers of pieces. When suggesting areas where future researchers could extend their work, the authors asked for a direct proof of the connection of their two-piece result to the hexagonal numbers, and a general formula for the number of divisions of an $m\times n$ board into $k$ pieces. Wagon \cite{wagon} answered the first question by using a combinatorial argument to relate the number of divisions to the pairs of numbers from the set $\{1,2,\dots,2n\}$.

Here, we answer the second question for a $2\times n$ rectangular grid. We achieve this by proving a recursive relationship that counts the number of divisions of the board into $k$ pieces. We then use fitting techniques on data generated from the recursion to find closed-form solutions for the number of ways to divide the board into four and five pieces. Finally, we show that any function satisfying this recursion must be a polynomial with predictable degree.

%Uncomment this line for "Defintion:..."

%\input{definitions}

%Uncomment these lines for definitions in-text:

We will use the following definitions throughout our discussion. A $2\times n$ \textbf{board} is a rectangle consisting of $1\times 1$ \textbf{squares}, consisting of $2$ squares vertically and $n$ squares horizontally. Two squares are \textbf{adjacent} if and only if they share an edge. We can represent the board as a collection of sets $B=(S, A)$, where the $S$ is the set of $2n$ squares $S=\{x_0,x_1,\dots,x_{2n-1}\}$, ordered top row then bottom row, left to right (see Figure~\ref{fig:labeling}), and $A$ is the set of pairs of adjacent squares. We can easily determine if a pair $(x_i,x_j)$ is in $A$ if $i+1=j$ or if $i+2=j$ (where $i<j$).

\begin{figure}
    \centering
    \includegraphics[width=4cm]{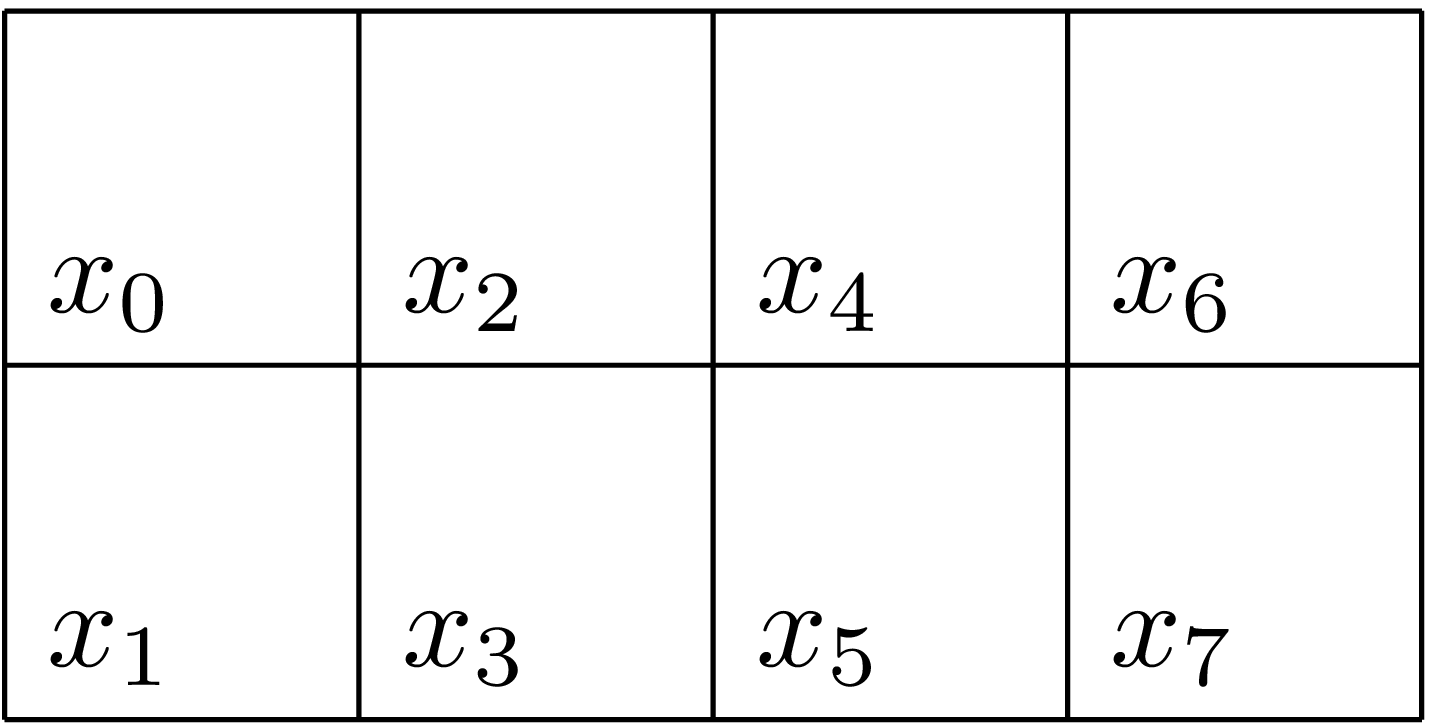}
    \caption{A $2\times 4$ board with labeled squares.}
    \label{fig:labeling}
\end{figure}

A \textbf{piece} is a subset $P\subseteq S$ such that for each square $x_p\in P$, there exists at least one other square $x_q\in P$ such that $x_p$ and $x_q$ are adjacent. A \textbf{division} $D$ of a board into $k$ pieces is a set of $k$ pieces $D=\{P_1,P_2,\dots,P_k\}$, such that each square is in exactly one piece (ie, $\bigcup_{i=1}^k P_i=S$, and $P_i\cap P_j=\emptyset$ if and only if $i\neq j$). We let $\mathcal{D}_k(n)$ denote the set of all divisions of the board into $k$ pieces, and we let $d_k(n)$ denote the number of divisions of a board into $k$ pieces, so $|\mathcal{D}_k(n)|=d_k(n)$. A \textbf{cut} is a pair of squares in a division that are adjacent, but are not in the same piece. 
%(ie, $(x_i,x_j)$ is a cut if and only if $x_i\in P\implies. x_j\notin P$).

Immediately, we can conclude some obvious results. As there is no way to divide the board into $0$ pieces, $d_0(n)=0$ for all $n$. It is also true that $d_1(n)=1$ for all $n$, as only the subset of $S$ that can be made is $D=\{S\}$. Likewise, $d_{2n}(n)=1$, as the only way to create $2n$ disjoint subsets whose union is $S$ is $D=\{\{x_0\},\{x_1\},\dots,\{x_{2n-1}\}\}$. 

%If we let $C\subseteq A$ be a set of cuts, then we can uniquely determine a division $D$ if and only if the following criterion is met: $(x_i,x_j)\in C$ implies there is no sequence of distinct squares $\{x_p\}_{p=0}^t$ such that $x_0=x_i$, $x_t=x_j$, and $(x_p,x_p+1)\in A\setminus C$ for $0\leq p\leq t-1$. This criterion states that any two adjacent squares that are separated by a cut cannot be connected by a path along adjacent squares that are not separated by cuts, which means they must be in separate pieces. 

%We let $\mathcal{C}(n)$ be the set of all sets of cuts $C$ that meet the above criterion. To show that each $C\in\mathcal{C}(n)$ uniquely determines a division $D\in \mathcal{D}(n)$, we show that there is a bijection between them. Let $f:\mathcal{C}(n)\to \mathcal{D}(n)$ be a function that maps each set of cuts to a division. The function is surjective because for each division, simply identify each pair of adjacent squares that are in different pieces, and place a cut between them. The function is injective because if two divisions have the same sets of pieces, then by surjectivity, each must have a set of cuts that determines it. Any additional cuts will result in another piece being added, or will cause the set to no longer meet the above criterion.

\begin{figure}
\centering
\includegraphics[width=4cm,keepaspectratio]{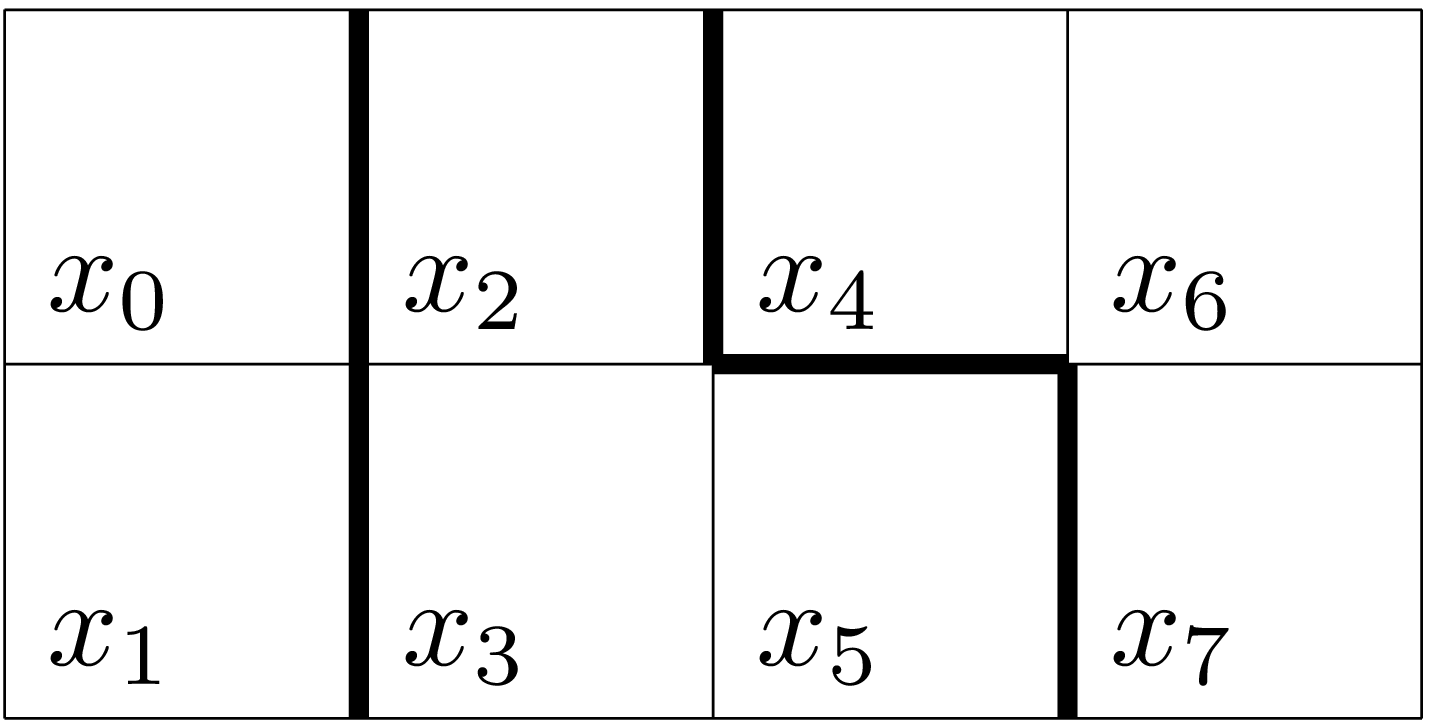}
\caption{A division of a $2\times 4$ board into 3 pieces. Expressed using our notation, this is $D= \{\{x_0,x_1\},\{x_2,x_3,x_5\},\{x_4,x_6,x_7\}\}$, and $D\in \mathcal{D}_3(4)$. Each thick line represents a cut.}
\label{fig:exampledivision}
\end{figure}

\section{Recursive Relationship for Counting Divisions into $k$ Pieces} %\label{recursion}
This problem is ripe for recursion because the $2\times n$ board can be extended by simply adding two squares onto the rightmost end, creating a $2\times (n+1)$ board. We can then use these two newly-added squares to form divisions of the $2\times (n+1)$ board in a way that builds off the divisions of the $2\times n$ board nested inside of it. We will denote the $2\times n$ board as $B=(S, A)$, where $S=\{x_0,x_1,\dots,x_{2n-1}\}$, and the $2\times (n+1)$ board as $B'=(S', A')$, where $S'=S\cup\{x_{2n},x_{2n+1}\}$.

In order for $B'$ to build upon the divisions of $B$, special attention needs to be paid to the two rightmost squares of $B$, $x_{2n-2}$ and $x_{2n-1}$. Whether or not these squares are in different pieces is important for how the two newly added squares can be added to the board to create more divisions. We will partition the set of divisions $\mathcal{D}_k(n)$ into $\mathcal{S}_k(n)$ and $\mathcal{T}_k(n)$, where $\mathcal{S}_k(n)$ is the set of divisions where $x_{2n-2}$ and $x_{2n-1}$ are not in the same piece, and $\mathcal{T}_k(n)$ is the set of divisions where $x_{2n-2}$ and $x_{2n-1}$ are in the same piece. We will denote the number of divisions in each set as $|\mathcal{S}_k(n)|=s_k(n)$ and $|\mathcal{T}_k(n)|=t_k(n)$. 
\begin{theorem}
The number of divisions of a $2\times n$ board into $k$ pieces satisfies the following recursion:
\begin{equation*}
d_k(n+1)=d_{k-2}(n)+3d_{k-1}(n)+d_k^2(n)+2s_k(n)
\end{equation*}
\label{drecursion}
\end{theorem}
\begin{proof}
The two additional squares in $B'$ can be used to add 0, 1, or 2 extra pieces to each division of $B$. Thus, to obtain a division into $k$ pieces for $B'$, we only have to consider how the additional squares interact with the divisions of $B$ into \textbf{(i)} $k-2$, \textbf{(ii)} $k-1$, and \textbf{(iii)} $k$ pieces. Throughout this proof, we let $D=\{P_1,P_2,\dots,P_\ell\}$ be the division in $B$ into $\ell$ pieces and we let $D'$ be the division in $B'$.

\textbf{Case (i):} 
%When $B$ is divided into $k-2$ pieces, each division is of the form $D=\{P_1,P_2,\dots,P_{k-2}\}$. 
To obtain $k$ pieces in $B'$, $x_{2n}$ and $x_{2n+1}$ must both be added as an individual piece, that is, the division in $B'$ is:\\ $D'=\{P_1,P_2,\dots,P_{k-2},\{x_{2n}\},\{x_{2n+1}\}\}$ (see Figure~\ref{fig:kminus2}).\\ 
There is only one way to add the two squares for each division $D\in \mathcal{D}_{k-2}(n)$ to obtain a division in $\mathcal{D}_k(n+1)$. Therefore, there must be a term of $d_{k-2}(n)$ in the recursion.

\textbf{Case (ii):} 
%When $B$ is divided into $k-1$ pieces, each division is of the form $D=\{P_1,P_2,\dots,P_{k-1}\}$. 
To obtain $k$ pieces in $B'$, $x_{2n}$ and $x_{2n+1}$ must be used to add exactly one piece. There are three ways to do this. The first is to add both as a single piece:\\
$D'=\{P_1,P_2,\dots,P_{k-1},\{x_{2n},x_{2n+1}\}\}$.\\
The other two ways are to add $x_{2n}$ or $x_{2n+1}$ as its own piece, and attach the other to piece containing $x_{2n-2}$ and $x_{2n-1}$, respectively. If $x_{2n-2}\in P_i$ and $x_{2n-1}\in P_j$ (note that it is \textit{not} necessary for $i\neq j$), then the two ways to add the single piece are:\\
$D'=\{P_1,P_2,\dots,P_i\cup\{x_{2n}\},\dots,P_{k-1},\{x_{2n+1}\}\}$ and\\ $D'=\{P_1,P_2,\dots,P_j\cup\{x_{2n+1}\},\dots,P_{k-1},\{x_{2n}\}\}$ (see Figure~\ref{fig:kminus1}).\\
Thus, for each division $D\in \mathcal{D}_{k-1}(n)$, there are 3 ways to obtain a division in $\mathcal{D}_k(n+1)$. Therefore, there must be a term of $3d_{k-1}(n)$ in the recursion.

\textbf{Case (iii):} 
%When $B$ is divided into $k$ pieces, each division is of the form $D=\{P_1,P_2,\dots,P_{k}\}$. 
To obtain $k$ pieces in $B'$, $x_{2n}$ and $x_{2n+1}$ must both be attached to some other piece in $B$. How these squares can be added depends on if the division in $B$ is in $\mathcal{T}_k(n)$ or $\mathcal{S}_k(n)$. 

If $D\in \mathcal{T}_k(n)$, then $x_{2n-2},x_{2n-1}\in P_i$, that is, they are in the same piece. The only way to add $x_{2n}$ and $x_{2n+1}$ to remain at $k$ pieces is to add them both to $P_i$, so:\\
$D'=\{P_1,P_2,\dots,P_i\cup\{x_{2n},x_{2n+1}\},\dots,P_k\}$.\\
There is only one way to add the two squares for each division $D\in \mathcal{T}_{k}(n)$ to obtain a division in $\mathcal{D}_k(n+1)$. Therefore, there must be a term of $t_{k}(n)$ in the recursion.

If $D\in \mathcal{S}_k(n)$, then $x_{2n-2}\in P_i$ and $x_{2n-1}\in P_j$, where $i\neq j$. There are three ways to add $x_{2n}$ and $x_{2n+1}$ to remain at exactly $k$ pieces. They can both be added to $P_i$, they can both be added to $P_j$, or they can split, with $x_{2n}$ being added to $P_i$ and $x_{2n+1}$ being added to $P_j$:\\
$D'=\{P_1,P_2,\dots,P_i\cup\{x_{2n},x_{2n+1}\},\dots,P_{k}\}$,\\
$D'=\{P_1,P_2,\dots,P_j\cup\{x_{2n},x_{2n+1}\},\dots,P_{k}\}$, and \\
$D'=\{P_1,P_2,\dots,P_i\cup\{x_{2n}\},P_j\cup\{x_{2n+1}\},\dots,P_{k}\}$ (see Figure~\ref{fig:sep_recursion}).\\
Thus, for each division $D\in \mathcal{S}_{k}(n)$, there are 3 ways to obtain a division in $\mathcal{D}_k(n+1)$. Therefore, there must be a term of $3s_{k}(n)$ in the recursion.

Adding everything together, we get the following:
\begin{equation*}
    d_k(n+1)=d_{k-2}(n)+3d_{k-1}(n)+t_k(n)+3s_k(n)
\end{equation*}
It is quite clear that $\mathcal{D}_k(n)=\mathcal{S}_k(n)\cup \mathcal{T}_k(n)$ and $\mathcal{S}_k(n)\cap \mathcal{T}_k(n)=\emptyset$, so $d_k(n)=s_k(n)+t_k(n)$. We can substitute this into the equation above to obtain the recursion in the theorem statement.
\end{proof}

\begin{figure}
    \centering
    \includegraphics[width=4cm]{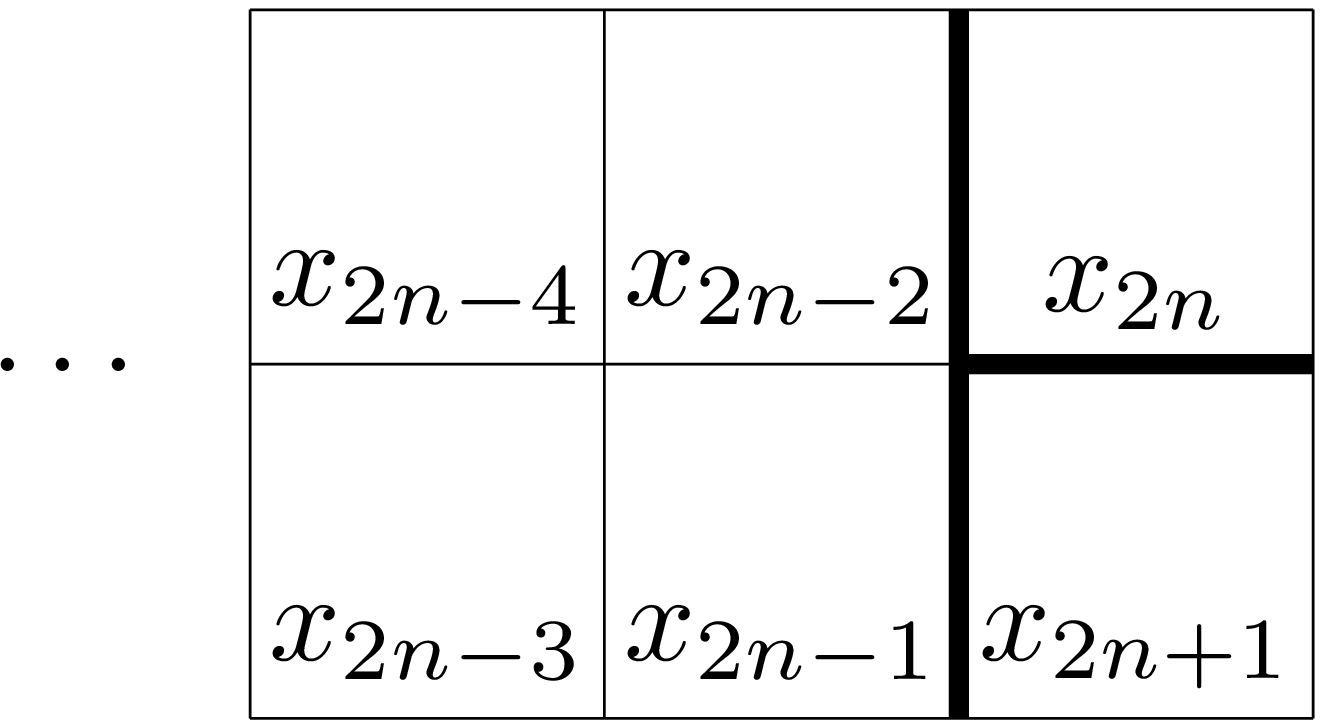}
    \caption{The single way to add the two squares to each division in $\mathcal{D}_{k-2}(n)$.}
    \label{fig:kminus2}
\end{figure}

\begin{figure}
    \centering
    \includegraphics[width=12cm]{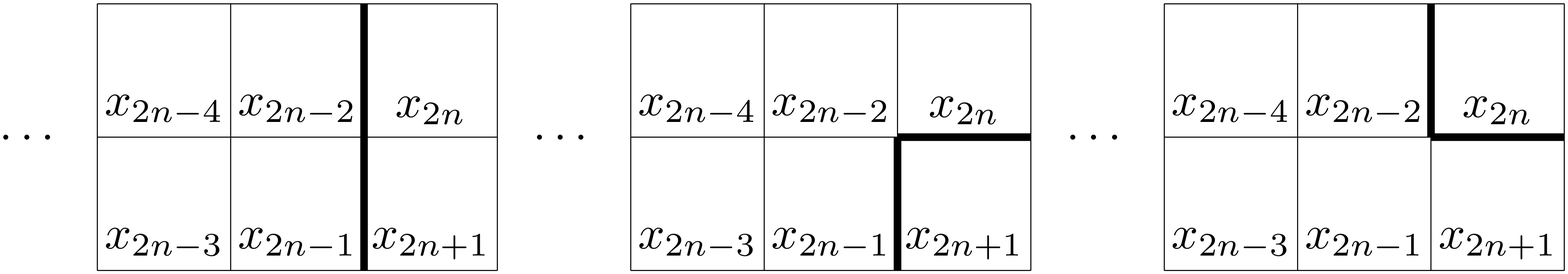}
    \caption{The three ways to add the two squares to each division in $\mathcal{D}_{k-1}(n)$. Note that a cut \textit{could} exist between $x_{2n-2}$ and $x_{2n-1}$, it has been omitted in this figure.}
    \label{fig:kminus1}
\end{figure}

\begin{figure}
    \centering
    \includegraphics[width=12cm]{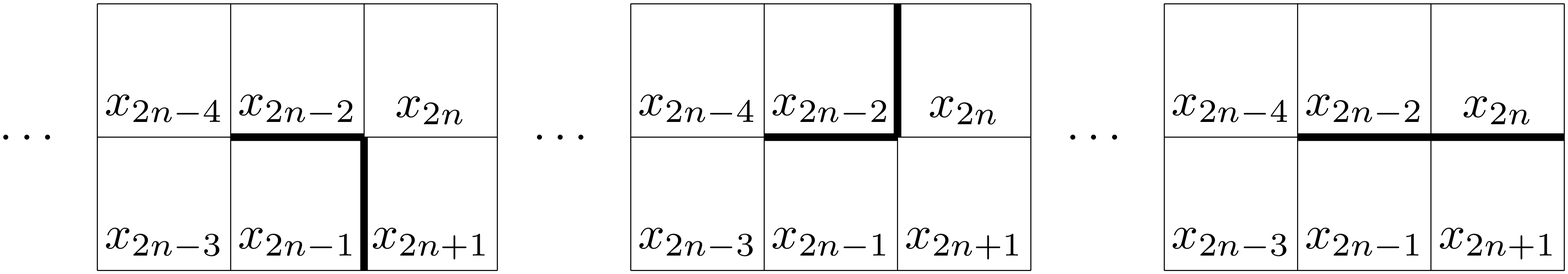}
    \caption{The three ways to add the two squares to each division in $\mathcal{S}_k(n)$.}
    \label{fig:sep_recursion}
\end{figure}

This recursion is only of use when $s_k(n)$ can be calculated. Fortunately, $s_k(n)$ can also be computed recursively, in terms that do not involve $d_k(n)$.
\begin{theorem}
The number of divisions of a $2\times n$ board into $k$ pieces where the rightmost squares separate satisfies the following recursion:
\begin{equation*}
s_k(n+1)=d_{k-2}(n)+2d_{k-1}(n)+s_k (n)
\end{equation*}
\label{s_recursion}
\end{theorem}
\begin{proof}
Using the same arguments from the proof of Theorem~\ref{drecursion}, finding these divisions simply involves counting the number of divisions of $B'$ where $x_{2n}$ and $x_{2n+1}$ end up in different pieces. They separate in the $d_{k-2}(n)$ term, in two of the cases of the $d_{k-1}(n)$ term, none of the cases of the $t_k(n)$ term, and in one of the cases of the $s_k(n)$ term, so the recursion is just the sum of these cases.
\end{proof}

As a check of this result, we show that Durham and Richmond's result for the number of divisions into two pieces satisfies this recursion. Note that our notation of $s_2(n)$ corresponds to their result that the number of partitions where the first two squares are in different pieces is $2n-1$.
\begin{align*}
d_0(n)=0,&\quad d_1(n)=1,\quad d_2(n)=2n^2-n,\quad s_2(n)=2n-1\\
d_2(n+1)&=0+3(1)+2n^2-n+2(2n-1)
\\&=3+2n^2-n+4n-2=2n^2+3n+1
\\&=2(n+1)^2-(n+1)
\end{align*}

\section{Closed-Form Solutions for Divisions of $2\times n$ Board into Three, Four, and Five Pieces} \label{sec:threeandfour}
Using the tools we have proven thus far, we can find closed-form solutions for the number of divisions of a $2\times n$ board into $k=3,4,5$ pieces. First, we find closed forms of $s_3(n)$, $s_4(n)$, and $s_5(n)$ by generating data using the recursions and then fitting a polynomial to the points. We repeat this process for $d_3(n)$ (and show it matches the result by Durham and Richmond), $d_4(n)$, and $d_5(n)$. Because we know the degrees of each of these polynomials, we only have to generate $2k-2$ points for each $s_k(n)$ and $2k-1$ points for each $d_k(n)$ to exactly determine their coefficients. 

There are multiple methods of interpolation, all of which will produce the same result. Our fitting methods are least-squares regression using Numpy~\cite{harris} and Lagrange interpolation~\cite{archer}. Newton's divided difference method~\cite{weisstein} is another avenue that future researchers can employ.

%Uncomment the one you want to see

%Formulas written as theorems and proofs:

%\input{formula_proofs.tex}

%Formulas summarized as a table:

\begin{table}[ht]
    \centering
    \begin{tabular}{c|c}
        $k$&$s_k(n)$\\\hline & \\
         3&$\frac{4}{3}n^3-3n^2+\frac{8}{3}n-1$ \\&\\
         4&$\frac{4}{15}n^5-\frac{4}{3}n^4+\frac{11}{3}n^3-\frac{37}{6}n^2+\frac{167}{30}n-2$ \\&\\
         5&$\frac{25}{1008}n^7-\frac{37}{180}n^6+\frac{71}{72}n^5-\frac{107}{36}n^4+\frac{751}{144}n^3-\frac{217}{45}n^2+\frac{149}{84}n$\\~
    \end{tabular}
    \caption{Closed-form solutions for the number of divisions into $k$ pieces where the rightmost squares separate.}
    \label{tab:s_formulas}
\end{table}

\begin{table}[ht]
    \centering
    \begin{tabular}{c|c}
        $k$&$d_k(n)$\\\hline & \\
         3&$\frac{2}{3}n^4-\frac{4}{3}n^3+\frac{11}{6}n^2-\frac{13}{6}n+1$ \\&\\
         4&$\frac{4}{45}n^6-\frac{2}{5}n^5+\frac{25}{18}n^4-\frac{7}{2}n^3+\frac{226}{45}n^2-\frac{18}{5}n+1$\\&\\
         5&$\frac{2}{315}n^8+\frac{2}{15}n^6-\frac{1}{5}n^5+\frac{49}{120}n^4-\frac{7}{12}n^3+\frac{1139}{2520}n^2-\frac{13}{60}n$\\~
    \end{tabular}
    \caption{Closed-form solutions for the number of divisions into $k$ pieces.}
    \label{tab:d_formulas}
\end{table}

The formulas for $s_k(n)$ are shown in Table \ref{tab:s_formulas} and the formulas for $d_k(n)$ are shown in Table \ref{tab:d_formulas}. These polynomials can be explicitly proven to hold by induction using the recursions. As an example, we show the formulas hold for $k=3$ by using the proposed polynomials, as well as the proven polynomials for $d_1(n)$ and $d_2(n)$.

\begin{example}
\begin{align*}
s_3(n+1)&=d_1(n)+2d_2(n)+s_3(n)\\
&=(1)+2n(2n-1)+\frac{4}{3}n^3-3n^2+\frac{8}{3}n-1\\
&=\frac{4}{3}n^3+n^2+\frac{2}{3}n\\
&=\frac{4}{3}(n+1)^3-3(n+1)^2+\frac{8}{3}(n+1)-1\\
d_3(n+1)&=d_1(n)+3d_2(n)+d_3(n)+2s_3(n)\\
&=(1)+3n(2n-1)+\frac{2}{3}n^4-\frac{4}{3}n^3+\frac{11}{6}n^2-\frac{13}{6}n+1\\
&\quad+2\left(\frac{4}{3}n^3-3n^2+\frac{8}{3}n-1\right)\\
&= \frac{2}{3}n^4+\frac{4}{3}n^3+\frac{11}{6}n^2+\frac{1}{6}n\\
&= \frac{2}{3}(n+1)^4-\frac{4}{3}(n+1)^3+\frac{11}{6}(n+1)^2-\frac{13}{6}(n+1)+1
\end{align*}
\end{example}
The proofs for $k=4$ and $k=5$ proceed in the exact same manner, and are omitted for brevity.

The greatest challenge to proving formulas like these is to actually obtain the polynomials; the proof that they hold consists of only algebraic manipulations that can readily be outsourced to a computer algebra system. In theory, this process can be continued for $k=6,7,\dots$, as long as enough points are generated using the recursion. So long as the computer calculating the fitting polynomials has sufficient precision, the decimals returned by the interpolation algorithm can be converted into fractions relatively easily. 

\section{Functional Forms of $d_k^2(n)$ and $s_k^2(n)$}\label{sec:forms}
In the previous section, we implicitly assumed that the functions that count the number of divisions are polynomials on $n$. This assumption is correct. In this section, we prove that for all $k$, $d_k(n)$ and $s_k(n)$ must be polynomials on $n$. Moreover, we show that the degrees of these polynomials are $2k-2$ and $2k-3$, respectively. 

Our proofs of the polynomial nature of the functions is based on the fact that the sum of polynomials is a polynomial, and our proof of their degrees is based on Faulhaber's formula. A straightforward derivation of Faulhaber's formula is given in \cite{orosi}. The formula gives an expression for the sum of the first $n$ integers each raised to the positive integer power $p$ as a polynomial of degree $p+1$. Stated more explicitly,
\begin{equation*}
    \sum_{k=1}^nk^p=\frac{1}{p+1}\sum_{j=0}^p(-1)^j\binom{p+1}{j} B_jn^{p+1-j},
\end{equation*}
where $B_j$ is the $j$-th Bernoulli number \cite{apostol}. As we are only interested in the degree of the polynomial, we will condense the coefficients into a single variable, 
\begin{equation*}
    \sum_{k=1}^nk^p=\sum_{j=0}^pA_jn^{p+1-j},
\end{equation*}
where $A_j=\frac{(-1)^j}{p+1}\binom{p+1}{j}B_j$. Before we can employ the formula, we must rewrite our recursions in terms of summations. Beginning with $s_k(n)$, by Theorem~\ref{s_recursion} we have
\begin{equation*}
s_k(n)=d_{k-2}(n-1)+2d_{k-1}(n-1)+s_k(n-1).
\end{equation*}
Substituting for $s_k(n-1)$ using the recursion gives
\begin{align*}
s_k(n)&=d_{k-2}(n-1)+2d_{k-1}(n-1)\\
&+d_{k-2}(n-2)+2d_{k-1}(n-2)+s_k(n-2).
\end{align*}
This can be repeated until the input is 0, where $d_k(0)=s_k(0)=0$ for all $k$, so
\begin{equation}
s_k(n)=\sum_{j=1}^{n-1} d_{k-2}(j)+2d_{k-1}(j).
\label{eq:s_summation}
\end{equation}
We can do the same for $d_k(n)$. From Theorem~\ref{drecursion} we have
\begin{equation*}
d_k(n)= d_{k-2}(n-1)+3d_{k-1}(n-1)+2s_k(n-1)+d_k(n-1).
\end{equation*}
Substituting for $d_k(n-1)$ using the recursion gives
\begin{align*}
d_k(n)&= d_{k-2}(n-1)+3d_{k-1}(n-1)+2s_k(n-1)\\
&+d_{k-2}(n-2)+3d_{k-1}(n-2)+2s_k(n-2)+d_k(n-2).
\end{align*}
Again, this can be repeated all the way to 0, so
\begin{equation}
d_k(n)=\sum_{j=1}^{n-1} d_{k-2}(j)+3d_{k-1}(j)+2s_k(j)
\label{eq:d_summation}
\end{equation}

\begin{theorem}
For all $k$, $s_k(n)$ and $d_k(n)$ are polynomials on $n$, where $s_k(n)$ is of degree $2k-3$ and $d_k(n)$ is of degree $2k-2$.
\end{theorem}
\begin{proof}
Proof is by strong induction on $k$

The base cases are satisfied for $k=1,2$, as $d_1(n)=1,d_2(n)=2n^2-n,s_1(n)=0,s_2=2n-1$, each of which is a polynomial with appropriate degree. Our inductive assumption is that for all positive integers $\ell$ less than some integer $k$, $s_\ell(n)$ and $d_\ell(n)$ are of degree $2\ell-3$ and $2\ell-2$, respectively. Using Equation~(\ref{eq:s_summation}), it must be a polynomial, because it is a sum of polynomials. The same holds true for $d_k(n)$ using Equation~(\ref{eq:d_summation}). To prove the degrees, we start with $s_k(n)$. From Equation~(\ref{eq:s_summation}), we have
\begin{equation*}
    s_k(n)=\sum_{j=1}^{n-1} d_{k-2}(j)+2d_{k-1}(j).
\end{equation*}
Expanding $d_{k-1}(j)$ as a polynomial of degree $2k-4$ gives
\begin{align*}
    s_k(n)&=\sum_{j=1}^{n-1}\left(d_{k-2}(j)+2\left(\sum_{i=0}^{2k-4} C_i j^i\right)\right)\\
    &=\sum_{j=1}^{n-1}\left( d_{k-2}(j)+2C_{2k-4}j^{2k-4}+2\left(\sum_{i=0}^{2k-5} C_i j^i\right)\right).
\end{align*}
By Faulhaber's formula, $\displaystyle\sum_{j=1}^{n-1} 2C_{2k-4}j^{2k-4}$ is equal to a polynomial of degree $2k-3$. The remainder of the summation involves terms of powers that are at most $2k-5$, so Faulhaber's formula will only produce polynomials of degree up to $2k-4$ for those terms. Therefore $s_k(n)$ is equal to a polynomial of degree $2k-3$. To prove $d_2(n)$ is equal to a polynomial of degree $2k-2$, we begin with Equation~(\ref{eq:d_summation}),
\begin{equation*}
d_k(n)=\sum_{j=1}^{n-1} d_{k-2}(j)+3d_{k-1}(j)+2s_k(j).
\end{equation*}
Expanding $s_k(j)$ as a polynomial of degree $2k-3$ gives
\begin{align*}
    d_k(n)&=\sum_{j=1}^{n-1}\left( d_{k-2}(j)+3d_{k-1}(j)+2\left(\sum_{i=0}^{2k-3}C_ij^i\right)\right)\\
    &=\sum_{j=1}^{n-1}\left( d_{k-2}(j)+3d_{k-1}(j)+2C_{2k-3}j^{2k-3}+2\left(\sum_{i=0}^{2k-4}C_ij^i\right)\right).
\end{align*}
By Faulhaber's formula, $\displaystyle\sum_{j=1}^{n-1} 2C_{2k-3}j^{2k-3}$ is equal to a polynomial of degree $2k-2$. The remainder of the summation involves terms of powers that are at most $2k-4$, so Faulhaber's formula will only produce polynomials of degree up to $2k-3$ for those terms. Therefore $d_k(n)$ is equal to a polynomial of degree $2k-2$. This completes the induction.
\end{proof}

The power in knowing that these functions are polynomials with predictable degrees is that their closed-form solutions can be found simply by using interpolation on data generated from the recursion, which can swiftly be implemented in most programming languages. A polynomial of degree $d$ is uniquely determined by $d+1$ points, so by generating enough data for some fixed $k$, the polynomial that counts the number of divisions can be calculated immediately, with no need for any combinatorial arguments.

\section{Future Work}
Counting the number of divisions of a $2\times n$ rectangular grid into $k$ pieces has exhibited multiple interesting patterns. The functions that count these divisions are all polynomials with predictable degrees, a fact that is not at all obvious from the first glance at this problem. 

While we have shown that all the division-counting functions will be polynomials, these polynomials are unknown to us for $k>5$. Our process of fitting polynomials to the data generated from the recursion eventually failed due to the limited precision of our computer. %Of course, doing the interpolations by hand will not result in rounding errors, the cost being that the algebra will become quite tedious.

Our recursion technique works well for the $2\times n$ board, but quickly becomes more complicated when the height of the board is 3 or greater. For a $3\times (n+1)$ board, the recursion that counts the number of divisions  will involve the number of divisions of the $3\times n$ board into $k-3,k-2,k-1$ and $k$ pieces. While the $k-3$ case is easy (just add the three new squares as separate pieces), the other cases will depend on how the three rightmost squares of the $3\times n$ board are distributed among the pieces of each division. The cases would be: (i) all three are in different pieces, (ii) two are in one piece, the last is in another (of which there are three subcases depending on which square is alone), and (iii) all three are in the same piece. This would require defining more sets than just $\mathcal{S}_k(n)$ and $\mathcal{T}_k(n)$ to account for all possible ways to add the new squares.

It's worth noting that this problem is deeply connected to graph theory. Our goal of counting the number of divisions of the $2\times n$ into $k$ pieces board is equivalent to counting the number of ways to remove edges from a $2\times n$ lattice graph while following some specific rules. Those rules being that the removed edges must result in the graph being decomposed into $k$ connected components, and that if an edge is removed between two vertices, those vertices must end up in different connected components, similar to how a cut determines that two adjacent squares end up in different pieces.

We were able to leverage this equivalence to write Python code that counts all the divisions of an $m\times n$ board into $k$ pieces for all $1\leq k\leq mn$. This was achieved by encoding the grid as an $m\times n$ lattice graph with connected squares being adjoined by edges. A cut corresponded to removing the edge between two vertices. Using Networkx~\cite{hagberg}, we were able to iterate through all possible edge removals and count the number of pieces that were returned for each removal that followed the rules mentioned previously. Removals that did not follow the rules were ignored. The program is not efficient, but it is exhaustive. To access the code, go to \url{https://github.com/jakebr118/Counting-Divisions/blob/main/division_counting}.

A pattern with our polynomials is that the leading coefficient appeared to be decreasing monotonically for greater and greater $k$. The leading coefficients of the polynomials for $k\leq 5$ are $2,\frac{2}{3},\frac{4}{45},\frac{2}{315}$, in order. Whether or not this pattern continues is unknown to us, but we conjecture that it holds. Our reasoning is that in order to ``control" the number of divisions at small $n$, the coefficient on the highest-order term must stay small. A rigorous proof (or counterexample) is definitely in order, though.

\begin{abstract}
What is the number of ways to divide a $2\times n$ rectangular grid, cutting along the grid lines, into $k$ pieces? Are there formulas that calculate these number of divisions? If so, what do they look like? We show that all division-counting formulas will be polynomials on $n$ for all $k$, and that for a fixed $k$, the division-counting polynomial will have degree $2k-2$. We find these polynomials up to $k=5$.
\end{abstract}
\textbf{Acknowledgements}\\
The author wishes to thank Emilie Wiesner and Daniel Visscher for their support, knowledge, and inspiration throughout the course of the project. 

\end{document}